\documentclass{article}

\usepackage{amsmath}
\usepackage{enumerate}
\usepackage{amssymb}
\usepackage{amsthm}
\usepackage{tikz-cd}
\usepackage{mathrsfs}
\usepackage{authblk}
\usepackage{tikz}
\usetikzlibrary{fit}

\newtheorem{thm}{Theorem}[section]

\newtheorem{lem}[thm]{Lemma}
\newtheorem{prop}[thm]{Proposition}
\theoremstyle{remark}
\newtheorem{rmk}[thm]{Remark}
\newtheorem*{clm}{Claim}

\theoremstyle{definition}

\newtheorem*{prob}{Problem}
\newtheorem*{conv}{Convention}

\DeclareMathOperator{\diff}{Diff}
\DeclareMathOperator{\Int}{Int}
\DeclareMathOperator{\sdiff}{SDiff}
\DeclareMathOperator{\im}{im}
\DeclareMathOperator{\id}{id}
\DeclareMathOperator{\vect}{Vect}
\DeclareMathOperator{\cat}{Cat}
\DeclareMathOperator{\pl}{PL}
\DeclareMathOperator{\rel}{\quad rel\ }

\begin{document}

\title{On the classification of certain $1$-connected $7$-manifolds and related problems}
\author{Xueqi Wang}
\affil{Department of Basic Sciences, Beijing International Studies University, Beijing 100024, P. R. China \authorcr Email: wangxueqi@amss.ac.cn}
\date{}
\maketitle

\begin{abstract}
  We study the classification of closed, smooth, spin, $1$-connected $7$-manifolds whose integral cohomology ring is isomorphic to $H^*(\mathbb{C}P^2\times S^3)$. We also prove that if the integral cohomology ring of a closed, smooth, spin, $1$-connected $7$-manifold is isomorphic to $H^*(\mathbb{C}P^2\times S^3)$ or $H^*(S^2\times S^5)$, this $7$-manifold admits a Riemannian metric with positive Ricci curvature.


\end{abstract}

\section{Introduction}

It is an interesting and important task to give a classification of a certain kind of manifold. In this paper, we are mainly concerned with the following problem:

\begin{prob}
Classify closed, smooth, spin, $1$-connected $7$-manifolds whose integral cohomology ring is isomorphic to $H^*(\mathbb{C}P^2\times S^3)$.
\end{prob}

Examples of such manifolds are easy to find. For any integer $k$, let $M'_{k,0}$ be the total space of the nonspin $S^3$-bundle over $\mathbb{C}P^2$ with Euler class $e=0$ and the first Pontryagin class $p_1=4k+1$. (We can identify $p_1$ with an integer since, for any generator $z\in H^2(\mathbb{C}P^2)$, $z^2$ gives the same generator in $H^4(\mathbb{C}P^2)$.) It is easy to see that $M'_{k,0}$ satisfies our condition. Nonspin $S^3$-bundles over $\mathbb{C}P^2$ have attracted the attention of geometers, as it has been proved in \cite{GZ11} that their total space admits Riemannian metrics with nonnegative sectional curvature. Relevant work concerning $7$-manifolds of the same homology type is \cite{Kre18}, where the manifolds are required to have their integral cohomology ring isomorphic to $H^*(S^2\times S^5\# S^3\times S^4)$.

Our main classification result is the following:

\begin{thm}\label{thm:10}
  Let $M$ be a closed, smooth, spin, $1$-connected $7$-manifold with integral cohomology ring isomorphic to $H^*(\mathbb{C}P^2\times S^3)$. Then
  \begin{enumerate}
    \item there exists some $k$ such that $M$ is PL-homeomorphic to $M'_{k,0}$. Moreover, the homotopy type of $M$ is uniquely determined by $p_1(M) \mod 24$, and the homeomorphism type and the PL-homeomorphism type are uniquely determined by $p_1(M)$;
    \item there exists some homotopy $7$-sphere $\Sigma^7$ such that $M$ is diffeomorphic to $M'_{k,0}\# \Sigma^7$.
  \end{enumerate}
\end{thm}

As an application of our result and the surgery theorem of Wraith (cf. Theorem \ref{thm:11} and \cite{Wra97}), we obtain:
\begin{thm}\label{thm:7}
  Let $M$ be a closed, smooth, spin, $1$-connected $7$-manifold with integral cohomology ring isomorphic to $ H^*(\mathbb{C}P^2\times S^3)$ or $H^*(S^2\times S^5)$. Then $M$ admits a Riemannian metric with positive Ricci curvature.
\end{thm}

This paper is organized as follows. In Section \ref{sec:4}, we give some notations and tools which will be used later. A concrete example, $S^3$-bundles over $\mathbb{C}P^2$, is discussed in Section \ref{sec:2}, which plays a key role in our classification result. The proof of our main classification result, Theorem \ref{thm:10}, is presented in Section \ref{sec:3}. Finally, in Section \ref{sec:6}, we discuss when certain $7$-manifolds admit metrics of positive Ricci curvature.

\section{Preliminaries}\label{sec:4}

In this section, we introduce some basic concepts and facts needed in this paper for the convenience of readers who are not familiar with them.

\subsection{Actions of homotopy sets}\label{sec:1}

In this subsection, all spaces are equipped with base points and all maps are assumed to preserve base points. We refer to \cite{Ark11,Swi75} for details on related concepts and results.

For a map $f: X\to Y$, consider the cofiber sequence
\[
Y\xrightarrow{l}C_f\xrightarrow{q}\Sigma X,
\]
where $C_f$ is the mapping cone of $f$, $\Sigma X$ is the suspension of $X$, $l$ is the inclusion, and $q$ is the quotient map. We define a map $\psi:C_f\to \Sigma X\vee C_f$ by
\[
\psi[y]=(*,[y]),
\]
\[
\psi[x,t]=
  \begin{cases}
    ([x,2t],*), & \mbox{if } 0\leq t\leq \frac{1}{2}, \\
    (*,[x,2t-1]), & \mbox{if } \frac{1}{2}\leq t\leq 1.
  \end{cases}
\]
For any space $Z$ together with maps $g:C_f\to Z$ and $a:\Sigma X\to Z$, define $g^a:C_f\to Z$ as the composition
\[
C_f\xrightarrow{\psi} \Sigma X\vee C_f \xrightarrow{(a, g)} Z.
\]
This defines an operation of $[\Sigma X,Z]$ on $[C_f,Z]$ by
\[
[g]^{[a]}=[g^a].
\]
Furthermore, when $Z$ is an H-space, the set $[C_f,Z]$ has a group structure.

\begin{prop}
  If $Z$ is an H-space, then $u^{\alpha}=q^*(\alpha)u=uq^*(\alpha)$ for all $\alpha\in [\Sigma X, Z]$ and $u\in[C_f,Z]$.
\end{prop}

The map $f:X\to Y$ induces the following exact sequence:
\[
\cdots \to [\Sigma Y,Z]\xrightarrow{(\Sigma f)^*}[\Sigma X,Z]\xrightarrow{q^*}[C_f,Z]\xrightarrow{l^*}[Y,Z].
\]

\begin{thm}
  \begin{enumerate}[$(1)$]
    \item Let $\rho, \sigma\in[C_f,Z]$. Then $l^*(\rho)=l^*(\sigma)$ if and only if $\sigma=\rho^{\gamma}$ for some $\gamma\in[\Sigma X,Z]$;
    \item Let $\gamma, \delta\in[\Sigma X,Z]$. Then $q^*(\gamma)=q^*(\delta)$ if and only if $\gamma=(\Sigma f)^*(\epsilon)+\delta$ for some $\epsilon\in[\Sigma Y, Z]$.
  \end{enumerate}
\end{thm}

\subsection{Homotopy groups of spheres}

We list some homotopy groups of spheres which will be used later (cf. \cite{Tod62}).

\begin{thm}
  \begin{enumerate}[$(1)$]
    \item $\pi_n(S^n)\cong\mathbb{Z}\{\iota_n\}$, where $\iota_n$ is represented by the identity map of $S^n$;
    \item $\pi_{n+1}(S^n)\cong
    \begin{cases}
      \mathbb{Z}\{\eta_2\}, & \mbox{if } n=2 \\
      \mathbb{Z}_2\{\eta_n\}, & \mbox{if } n\geq 3,
    \end{cases}$

    where $\eta_2$ is represented by the Hopf map $S^3\to S^2$ and $\eta_n=\Sigma^{n-2}\eta_2$;
    \item $\pi_{n+2}(S^n)\cong\mathbb{Z}_2\{\eta_n\eta_{n+1}\}$ for $n\geq 2$;
    \item $\pi_{n+3}(S^n)\cong
    \begin{cases}
      \mathbb{Z}_2\{\eta_2\eta_3\eta_4\} , & \mbox{if } n=2 \\
      \mathbb{Z}_{12}\{a_3\}, & \mbox{if } n=3 \\
      \mathbb{Z}_{12}\{\Sigma a_3\}\oplus\mathbb{Z}\{\nu_4\}, & \mbox{if } n=4 \\
      \mathbb{Z}_{24}\{\nu_n\}, & \mbox{if } n\geq 5,
    \end{cases}$

    where $\nu_4$ is represented by the Hopf map $S^7\to S^4$, $\nu_n=\Sigma^{n-4}\nu_4$, $\eta_3\eta_4\eta_5=6a_3$, and $\Sigma^2a_3=2\nu_5$.
  \end{enumerate}
\end{thm}

\subsection{Surgery theory}

Recall the simply connected version of the surgery exact sequence due to Browder, Novikov, Sullivan, Wall, and Kirby-Seibenmann:
\[
\cdots \to L_{n+1}(\mathbb{Z})\xrightarrow{\omega} \mathscr{S}^{\cat}(M) \xrightarrow{\eta} [M,G/\cat] \xrightarrow{\theta} L_n(\mathbb{Z}).
\]
We will explain it briefly. Further details can be found in \cite{Ran02,Wal69}.

\begin{itemize}
  \item The category $\cat$ can be either the topological (TOP), piecewise linear (PL), or smooth (O) category. The space $M$ is a simply connected, closed, $n$-dimensional $\cat$-manifold.
  \item $L_n(\mathbb{Z})\cong
  \begin{cases}
    \mathbb{Z}, & \mbox{if } n\equiv 0 \mod 4 \\
    \mathbb{Z}_2, & \mbox{if } n\equiv 2 \mod 4 \\
    0, & \mbox{if $n$ is odd}.
  \end{cases}$
  \item A $\cat$-manifold structure $(N,f)$ on $M$ is an $n$-dimensional $\cat$-manifold $N$ together with a homotopy equivalence $f:N\to M$. The $\cat$-structure set $\mathscr{S}^{\cat}(M)$ of $M$ is the set of equivalence classes of $\cat$-manifold structures $(N,f)$, subject to the equivalence relation: $(N,f)\sim (N',f')$ if there exists a $\cat$-isomorphism $g:N\to N'$ such that $f$ is homotopy equivalent to $f'\circ g$.
  \item The space $G/\cat$ denotes the homotopy fiber of $B\cat\to BG$. Homotopy classes of maps $M\to G/\cat$ are in 1-1 correspondence with the equivalence classes of pairs $(\eta,t)$, where $\eta$ is a stable $\cat$-bundle over $M$ and $t$ is a fiber homotopy trivialization of $J\eta$, with $J\eta$ being $\eta$ regarded as a spherical fibration. The map $[M,G/\cat]\to [M,B\cat]$ can be interpreted by forgetting the trivialization.
  \item A homotopy equivalence $f:N\to M$ of $n$-dimensional $\cat$-manifolds can determine a fiber homotopy trivialization $t(f)$ of $J(\nu_M-(f^{-1})^*(\nu_N))$ in a standard way, where $\nu_M$ denotes the $\cat$-normal bundle of $M$ and $f^{-1}$ is the homotopy inverse of $f$. With the above interpretation of $[M,G/\cat]$, we have
      $$
      \eta(f)=(\nu_M-(f^{-1})^*(\nu_N),t(f)).
      $$
\end{itemize}

Then we list some facts which will be used later. We refer to \cite{Sul96} for further details.
\begin{enumerate}
  \item For $\cat=\pl$, $\eta$ is injective.
  \item If $W$ is a $\cat$-manifold with boundary, then the $\cat$-structure set $\mathscr{S}^{\cat}(W)$ can be similarly defined as
      $$
      \mathscr{S}^{\cat}(W)=\{f:(V,\partial V)\xrightarrow{\simeq} (W,\partial W)\}/\sim.
      $$
      Here $V$ is a $\cat$-manifold with boundary $\partial V$ and $f:(V,\partial V)\xrightarrow{\simeq} (W,\partial W)$ is a homotopy equivalence of pairs. The relation $(V,f)\sim (V',f')$ holds if there exists a $\cat$-isomorphism $g:V\to V'$ such that $f$ is homotopic to $f'\circ g$ as a map of pairs. The map $\eta:\mathscr{S}^{\cat}(W)\to  [W,G/\cat]$ can be similarly interpreted as above and is an isomorphism when $\pi_1(W)=\pi_1(\partial W)=1$. We also have the following commutative diagram:
      \[
      \begin{tikzcd}
      \mathscr{S}^{\cat}(W) \arrow[r,"\eta"] \arrow[d,"i^*"] & {[W, G/\cat]} \arrow[d,"i^*"] \\
      \mathscr{S}^{\cat}(\partial W) \arrow[r,"\eta"] & {[\partial W,G/\cat]}
      \end{tikzcd}
      \]
      where $i:\partial W\to W$ is the inclusion map and $i^*:\mathscr{S}^{\cat}(W) \to \mathscr{S}^{\cat}(\partial W)$ is given by restricting a homotopy equivalence of pairs to the boundary.
  \item $\pi_n(G/\pl)\cong
  \begin{cases}
    \mathbb{Z}, & \mbox{if } n\equiv 0 \mod 4 \\
    \mathbb{Z}_2, & \mbox{if } n\equiv 2 \mod 4 \\
    0, & \mbox{if $n$ is odd}.
  \end{cases}$
\end{enumerate}

\subsection{Smoothing theory}

Smoothing theory is concerned with the problem of finding and classifying smoothings of PL-manifolds. We will give a review of what is needed later. For more details, one can refer to \cite{HM74}.

Let $M$ be a PL-manifold. A smoothing of $M$ is a pair $(N,f)$, where $N$ is a smooth manifold and $f:N\to M$ is a PL-homeomorphism. Two smoothings of $M$ are concordant if they extend to a smoothing of $M\times I$. Let $\mathcal{S}^{\pl/O}(M)$ be the set of concordance classes of smoothings of $M$.

\begin{thm}[{\cite[II.4.2]{HM74}}]
  $\mathcal{S}^{\pl/O}(M)\cong [M,\pl/O]$, where $\pl/O$ is the homotopy fiber of $BO\to B\pl$.
\end{thm}

This theorem implies that $\pi_n(\pl/O)$ is isomorphic to $\mathcal{S}^{\pl/O}(S^n)$. Together with the generalized Poincar\'e conjecture and h-cobordism theorem, we have $\pi_n(\pl/O)\cong\Theta_n$ ($n\geq 6$), where $\Theta_n$ is the group of homotopy $n$-spheres. It is known that $\pl/O$ is $6$-connected (cf. \cite[p.123]{HM74}).

\section{Examples: $S^3$-bundles over $\mathbb{C}P^2$}\label{sec:2}

In this section, we will discuss the properties of a special kind of manifold. These manifolds are total spaces of linear $S^3$-bundles over $\mathbb{C}P^2$.

To begin with, we review some basic facts of $4$-dimensional vector bundles over $\mathbb{C}P^2$. Let $\vect_{\mathbb{R}}^{n,+}(X)$ be the set of isomorphism classes of $n$-dimensional oriented real vector bundles over a space $X$. It is well-known that $\vect_{\mathbb{R}}^{n,+}(X)\cong [X,BSO(n)]$. In particular, $4$-dimensional oriented vector bundles over $\mathbb{C}P^2$ are in 1-1 correspondence with elements in $[\mathbb{C}P^2, BSO(4)]$. For this homotopy set, we have the following exact sequence:
\[
[S^4,BSO(4)]\rightarrow [\mathbb{C}P^2,BSO(4)] \rightarrow [S^2,BSO(4)].
\]
We have $\vect_{\mathbb{R}}^{4,+}(S^4)\cong \pi_4(BSO(4))\cong \mathbb{Z}\{\alpha\}\oplus\mathbb{Z}\{\beta\}$. The bundles $\alpha,\beta$ can be chosen to satisfy the following condition:
\begin{center}
  \begin{tabular}{c|c|c}
   & $\alpha$ & $\beta$ \\
  \hline
  $e$ & $0$ & $\omega_{S^4}$ \\
  \hline
  $p_1$ & $4\omega_{S^4}$ & $-2\omega_{S^4}$\\
\end{tabular}
\end{center}
where $\omega_{S^4}$ is the orientation cohomology class of $S^4$. It follows that a $4$-dimensional oriented vector bundle over $S^4$ is determined up to isomorphism by its first Pontryagin class and Euler class. We also have $\vect_{\mathbb{R}}^{4,+}(S^2)\cong \pi_2(BSO(4))\cong \mathbb{Z}_2$, which means there are two oriented $4$-dimensional vector bundles over $S^4$ up to isomorphism. These two bundles can be distinguished by their second Stiefel-Whitney classes, and both have extensions over $\mathbb{C}P^2$: the trivial one extends to $\epsilon^4$, and the nontrivial one extends to $\gamma\oplus\epsilon^2$, where $\gamma$ is the tautological line bundle over $\mathbb{C}P^2$. Combining these facts, it is not difficult to see that a $4$-dimensional vector bundle over $\mathbb{C}P^2$ is determined up to isomorphism by its first Pontryagin class, Euler class, and second Stiefel-Whitney class (cf. \cite{DW59}), and can be constructed as follows:
\[
\begin{matrix}
  \text{Spin case:} & \quad & \text{Nonspin case:} \\
  \begin{tikzcd}
\xi_{k,l} \arrow[r] \arrow[d] & k\alpha+l\beta \arrow[d] \\
\mathbb{C}P^2 \arrow[r,"q"] & S^4
\end{tikzcd} & \quad &
\begin{tikzcd}
\xi'_{k,l} \arrow[r] \arrow[d] & (\gamma\oplus\epsilon^2)\vee (k\alpha+l\beta) \arrow[d] \\
\mathbb{C}P^2 \arrow[r,"v"] & \mathbb{C}P^2 \vee S^4
\end{tikzcd}
\end{matrix}
\]
where $q$ and $v$ are the obvious collapsing maps. Using the additivity of the first Pontryagin class and the Euler class under Whitney sum of $4$-dimensional vector bundles over $S^4$, we have
\begin{align*}
  p_1(\xi_{k,l}) & =(4k-2l)\omega_{\mathbb{C}P^2}, & e(\xi_{k,l}) & =l\omega_{\mathbb{C}P^2}, \\
  p_1(\xi'_{k,l}) & =(4k-2l+1)\omega_{\mathbb{C}P^2}, & e(\xi'_{k,l}) & =l\omega_{\mathbb{C}P^2},
\end{align*}
where $\omega_{\mathbb{C}P^2}$ is the orientation cohomology class of $\mathbb{C}P^2$.

We now turn to the discussion of our examples. Let $M_{k,l}=S(\xi_{k,l})$ and $M'_{k,l}=S(\xi'_{k,l})$, where $S(\xi)$ is the associated sphere bundle of the vector bundle $\xi$. Then $M_{k,l}$ and $M'_{k,l}$ are $1$-connected $7$-manifolds whose cohomology ring satisfies the following condition:
\begin{enumerate}
  \item isomorphic to $H^*(\mathbb{C}P^2\times S^3)\cong \mathbb{Z}[x,y]/(x^3,y^2)$, if $l=0$;
  \item $H^2\cong \mathbb{Z}\{u\}$, $H^3=0$, $H^4\cong\mathbb{Z}_l\{u^2\}$, if $l\neq 0$.
\end{enumerate}
In addition, $w_2(M'_{k,l})=0$ and $w_2(M_{k,l})\neq 0$.

For the classification of $M_{k,l}$ and $M'_{k,l}$ when $l\neq 0$, one can refer to \cite{EZ14}. We point out that only a partial result has been obtained for the homotopy classification since the homotopy classification of $7$-manifolds satisfying condition (2) is incomplete, see \cite{Kru97,Kru98}. However, things are much easier for the case $l=0$.

\begin{thm}\label{thm:3}
  \begin{enumerate}
    \item Two manifolds $M_{k,0}$ and $M_{k',0}$ are homeomorphic or diffeomorphic if and only if $k=k'$;
    \item Two manifolds $M_{k,0}$ and $M_{k',0}$ are homotopy equivalent if and only if $k\equiv k' \mod 6$.
  \end{enumerate}
  These two statements also hold when $M_{k,0}$ and $M_{k',0}$ are replaced by $M'_{k,0}$ and $M_{k',0}'$, respectively.
\end{thm}

\begin{proof}
  \begin{enumerate}
    \item Easy calculation shows that $p_1(M_{k,0})=(4k+3)x^2$ and $p_1(M'_{k,0})=(4k+4)x^2$. Since Novikov proved that rational Pontryagin classes are homeomorphism invariants \cite{Nov65}, the first statement follows.
    \item The only if part follows from a theorem of Hirzebruch \cite[Theorem 4.5]{Hir95}, which states that the first Pontryagin class modulo torsion is a homotopy invariant mod $24$.

        The if part is due to the following stronger proposition (Proposition \ref{thm:9}).
  \end{enumerate}
\end{proof}

\begin{prop}\label{thm:9}
  The $S^3$-bundle $M_{k,0}\to \mathbb{C}P^2$ is fiber homotopy equivalent to $M_{k',0}\to \mathbb{C}P^2$ if $k\equiv k' \mod 6$. The statement also holds for $M_{k,0}'\to \mathbb{C}P^2$.
\end{prop}

\begin{proof}
  The bundles $\xi_{k,0}$ and $\xi'_{k,0}$ can be viewed as elements in $[\mathbb{C}P^2, BSO(3)]$ since their structure group can be reduced to $SO(3)$. Therefore, the proposition will follow if we can show that the images of $\xi_{k,0}$ (resp. $\xi'_{k,0}$) and $\xi_{k',0}$ (resp. $\xi'_{k',0}$) under $j_*:[\mathbb{C}P^2, BSO(3)]\to[\mathbb{C}P^2, BF(4)]$ are the same if $k\equiv k' \mod 6$. Here $F(m)=\{f:S^{m-1}\to S^{m-1}| \deg f=1, f(*)=*\}$ with $*$ the base point and $j:BSO(3)\to BF(4)$ is the map induced by the inclusion $SO(3)\subset F(4)$.

  Our calculation relies on the following exact ladder of Puppe sequences:
  \[
  \begin{tikzcd}[row sep=small, column sep=tiny]
   & \mathbb{Z}\{\alpha\} \arrow[d,equal] & \{\xi_{k,0},\xi'_{k,0}\} \arrow[d,equal] & \mathbb{Z}_2\{\gamma'\} \arrow[d,equal] & \\
  0 \arrow[r] \arrow[dd] & {[S^4,BSO(3)]} \arrow[r,"q^*"] \arrow[dd,"j_*","\text{surj.}"'] & {[\mathbb{C}P^2,BSO(3)]} \arrow[r,"i^*"] \arrow[dd,"j_*"] & {[S^2,BSO(3)]} \arrow[r] \arrow[dd,"j_*","\cong"'] & 0 \arrow[dd] \\
   & & & & \\
  {[S^3,BF(4)]} \arrow[r,"\eta_3^*","\text{inj.}"'] \arrow[dd,"e_*","\cong"'] & {[S^4,BF(4)]} \arrow[r,"q^*"] \arrow[dd,"e_*","\text{inj.}"'] & {[\mathbb{C}P^2,BF(4)]} \arrow[r] \arrow[dd,"e_*","\text{inj.}"'] & {[S^2,BF(4)]} \arrow[r,"\eta_2^*=0"] \arrow[dd,"e_*","\cong"'] & {[S^3,BF(4)]} \arrow[dd,"e_*","\cong"'] \\
   & & & & \\
  {[S^3,BF]} \arrow[r,"\eta_3^*","\text{inj.}"'] & {[S^4,BF]} \arrow[r,"q^*"] & {[\mathbb{C}P^2,BF]} \arrow[r] & {[S^2,BF]} \arrow[r,"\eta_2^*=0"] & {[S^3,BF]}
  \end{tikzcd}
  \]
  We provide some explanation of this diagram.
  \begin{enumerate}
    \item The first row is exact since $\pi_3(BSO(3))=0$.
    \item It is well known that $j_*: \pi_i(BSO(3))\to \pi_i(BF(4))$ is an isomorphism for $i=2$ and surjective for $i=4$, where $\pi_2(BF(4))\cong \mathbb{Z}_2$ and $\pi_4(BF(4))\cong\pi_6(S^3)\cong\mathbb{Z}_{12}$.
    \item The vertical maps from the second row to the third row are induced by the inclusion-induced map $e:BF(4)\to BF$. Note that the third row is an exact sequence of abelian groups, as $BF$ is an infinite loop space \cite{BV68}. Since $e_*:\pi_i(BF(4))\to \pi_i(BF)$ coincides with the suspension map $\Sigma:\pi_{i+2}(S^3)\to\pi_{i-1}^s$, it follows that $e_*$ is an isomorphism for $i=2,3$ and injective for $i=4$.
    \item The reason for the injectivity of $e_*:[\mathbb{C}P^2,BF(4)]\to[\mathbb{C}P^2,BF]$ is as follows. Let $[g_1],[g_2]\in[\mathbb{C}P^2,BF(4)]$ such that $e_*[g_1]=e_*[g_2]$, i.e. $eg_1\simeq eg_2:\mathbb{C}P^2\to BF$. The obstructions of lifting the homotopy to $BF(4)$ lie in $H^i(\mathbb{C}P^2;\pi_i(F/F(4)))$, where $F/F(4)$ is the homotopy fiber of $e:BF(4)\to BF$. It is easy to see that $\pi_i(F/F(4))=0$ for $i=2,4$. Since $H^i(\mathbb{C}P^2;G)=0$ ($i=1,3$) for any coefficient $G$, there is no obstruction and we have $g_1\simeq g_2$. The injectivity of $e_*$ follows.
    \item Now it is clear that $\eta_2^*=0$ in the second and third row. As $\eta_3^*:\pi_3(BF(4))\to \pi_4(BF(4))$ coincides with $\eta_5^*:\pi_5(S^3) \to \pi_6(S^3)$, $\eta_3^*$ is injective and $\im \eta_3^*=\mathbb{Z}_2\{6j_*\alpha\}$. It is similar for the third row.
    \item The construction of $\xi_{k,0}$ and $\xi'_{k,0}$ implies that $\xi_{k,0}=kq^*\alpha$ and $\xi'_{k,0}={\xi'_{0,0}}^{k\alpha}$ (refer to Subsection \ref{sec:1} for the notations). In addition, $\xi'_{0,0}=\gamma\oplus\epsilon^2$.
  \end{enumerate}

  We now turn to the calculation. For any integer $m$,
  \[
  j_*(\xi_{k+6m,0})=j_*((k+6m)q^*\alpha)=(k+6m)q^*j_*\alpha=kq^*j_*\alpha,
  \]
  where the last equation holds since $6mj_*\alpha\in\im \eta_3^*=\ker q^*$. In addition, $kq^*j_*\alpha=j_*(k q^*\alpha)=j_*(\xi_{k,0})$.

  For $\xi'_{k,0}$, we have
  \begin{align*}
    e_*j_*(\xi'_{k+6m,0}) & =e_*j_*({\xi'_{0,0}}^{(k+6m)\alpha})=(e_*j_*(\xi'_{0,0}))^{(k+6m)e_*j_*\alpha} \\
     & =e_*j_*(\xi'_{0,0})+q^*((k+6m)e_*j_*\alpha) \\
     & =e_*j_*(\xi'_{0,0})+kq^*e_*j_*\alpha=e_*j_*(\xi'_{k,0}).
  \end{align*}
  It follows from the injectivity of $e_*$ that $j_*(\xi'_{k+6m,0})=j_*(\xi'_{k,0})$. It is now obvious that the proposition holds.
\end{proof}

\section{Proof of Theorem \ref{thm:10}}\label{sec:3}

\subsection{Classification up to homotopy equivalence}

Let $M$ be a closed, smooth, $1$-connected $7$-manifold with integral cohomology ring isomorphic to $H^*(\mathbb{C}P^2\times S^3)$, i.e. $H^*(M)\cong \mathbb{Z}[x,y]/(x^3,y^2)$ with $\deg x=2$ and $\deg y=3$. Then $M$ has a minimal cell structure up to homotopy equivalence (cf. \cite[Proposition 4C.1]{Hat02}):
\[
M\simeq S^2\cup_{\phi_3}e^3\cup_{\phi_4}e^4\cup_{\phi_5}e^5\cup_{\phi_7}e^7.
\]
Let $M^{(n)}$ be the $n$-skeleton of this minimal cell structure. Before analyzing the attaching maps, we make the following convention:

\begin{conv}
For $A\subset X$, the image of an element $a\in \pi_i(A)$ under the inclusion-induced homomorphism $\pi_i(A)\to\pi_i(X)$ will also be denoted by $a$.
\end{conv}

We now start the analysis of $\phi_i$.

  1. It follows from $H_2(M)\cong \mathbb{Z}$ that $\phi_3=0$, namely, $M^{(3)}=S^2\vee S^3$.

  2. We have $\pi_3(M^{(3)})\cong\pi_3(S^2\vee S^3)\cong\mathbb{Z}\{\eta_2\}\oplus\mathbb{Z}\{\iota_3\}$. As $H_3(M)\cong \mathbb{Z}$, the component of $\phi_4$ on $\iota_4$ must be $0$. Together with the fact that $x^2$ generates $H^4(M)\cong \mathbb{Z}$, we can see that $\phi_4=\pm\eta_2$. As attaching $e^4$ using $\eta_2$ or $-\eta_2$ will lead to homotopy equivalent complexes, we may assume that $\phi_4=\eta_2$. (Similar problems may occur below, and we will not explain again.) Therefore $M^{(4)}=\mathbb{C}P^2\vee S^3$.

  3. We have $\pi_4(M^{(4)})\cong\pi_4(\mathbb{C}P^2\vee S^3)\cong\mathbb{Z}\{[\iota_2,\iota_3]\}\oplus\mathbb{Z}_2\{\eta_3\}$, where $[\cdot,\cdot]$ denotes the Whitehead product. Since $xy$ generates $H^5(M)\cong \mathbb{Z}$, the coefficient of $\phi_5$ on $[\iota_2,\iota_3]$ must be $\pm 1$. Therefore we can assume $\phi_5=[\iota_2,\iota_3]+\delta\eta_3$, where $\delta=0$ or $1$.

      \begin{lem}\label{thm:1}
        Let $M$ and $\delta$ be as above. Then $w_2(M)=0$ if and only if $\delta=1$.
      \end{lem}

      \begin{proof}
        The condition $\delta=1$ is equivalent to $Sq^2\neq 0:H^3(M;\mathbb{Z}_2)\to H^5(M;\mathbb{Z}_2)$. Let $\overline{a}$ denote the mod $2$ reduction of $a\in H^*(M)$. Since $\overline{y}$ is the only nontrivial element in $H^3(M;\mathbb{Z}_2)\cong\mathbb{Z}_2$, it suffices to analyze if $Sq^2(\overline{y})$ is trivial or not. Our method is to compare $Sq^2:H^5(M;\mathbb{Z}_2)\to H^7(M;\mathbb{Z}_2)$ computed by two methods. On one hand, $Sq^2(\overline{xy})=w_2(M)\overline{xy}$ (cf. \cite{MS74}). On the other hand, $Sq^2(\overline{xy})=Sq^2(\overline{x})\overline{y}+Sq^1(\overline{x})Sq^1(\overline{y})+\overline{x}Sq^2(\overline{y})=\overline{x^2y}+\overline{x}Sq^2(\overline{y})$. Therefore, $Sq^2(\overline{y})=w_2(M)\overline{y}+\overline{xy}$, which indicates that $Sq^2(\overline{y})\neq 0$ if and only if $w_2(M)=0$. We have thus proved the lemma.
      \end{proof}

  4. The discussion of $\phi_7$ will be divided into two parts: the spin case and the nonspin case.

    (1) The spin case, i.e. $\delta=1$ and $M^{(5)}=(\mathbb{C}P^2\vee S^3)\cup_{[\iota_2,\iota_3]+\eta_3}e^5$.

        \begin{lem}\label{thm:12}
          If $M$ is spin, then $\pi_6(M^{(5)})\cong\mathbb{Z}\{\beta_1\}\oplus\mathbb{Z}_{12}\{a_3\}\oplus\mathbb{Z}_2\{p\eta_5\}$, where $\beta_1$ is the attaching map of the top cell of $M'_{0,0}$ equipped with the standard cell structure and $p:S^5\to \mathbb{C}P^2$ is the Hopf fibration.
        \end{lem}

        \begin{proof}
          Since $M^{(5)}\simeq (M'_{0,0})^{(5)}$, we only need to calculate $\pi_6((M'_{0,0})^{(5)})$.

          Consider the homotopy exact sequence of the pair $(M'_{0,0},(M'_{0,0})^{(5)})$:
          \[
          \cdots\to \pi_{i+1}(M'_{0,0},(M'_{0,0})^{(5)})\to\pi_i((M'_{0,0})^{(5)})\to\pi_i(M'_{0,0})\to\cdots.
          \]
          This sequence splits into short exact sequences:
          \[
          0\to \pi_{i+1}(M'_{0,0},(M'_{0,0})^{(5)})\to\pi_i((M'_{0,0})^{(5)})\to\pi_i(M'_{0,0})\to 0.
          \]
          To see this, first observe that there is an isomorphism $\sigma_*+i_*:\pi_i(\mathbb{C}P^2)\oplus\pi_i(S^3)\to\pi_i(M'_{0,0})$, where $\sigma$ is a section of $S(\xi'_{0,0})$ and $i$ is the inclusion of fiber. Since images of $\sigma$ and $i$ lie in $(M'_{0,0})^{(5)}$, it follows that $\sigma_*+i_*$ factors through $\pi_i((M'_{0,0})^{(5)})$, which gives a splitting map $\pi_i(M'_{0,0})\to \pi_i((M'_{0,0})^{(5)})$.
          It is now obvious that the lemma holds.
        \end{proof}

        To see the constraints on $\phi_7$, consider the $S^1$-bundle over $M$ with Euler class $x$: $S^1\to N\to M$. It is easy to show that $\pi_1(N)=1$ and $H_*(N)\cong H_*(S^3\times S^5)$, which implies that $N\simeq S^3\cup e^5\cup e^8$. Since $\pi_4(N)\cong \pi_4(M)\cong\pi_4(M'_{0,0})\cong\pi_4(\mathbb{C}P^2)\oplus\pi_4(S^3)\cong\mathbb{Z}_2$, we must have $N\simeq (S^3\vee S^5)\cup e^8$. It follows that $\pi_6(M)\cong\pi_6(N)\cong\pi_6(S^3\vee S^5)\cong\mathbb{Z}_{12}\oplus\mathbb{Z}_2$. Combining the fact that $\pi_6(M)\cong\pi_6(M^{(5)})/(\phi_7)$, we see that the coefficient of $\phi_7$ on $\beta_1$ must be $\pm 1$, which can always be chosen as $+1$.

    (2) The nonspin case, i.e $\delta=0$ and $M^{(5)}=(\mathbb{C}P^2\vee S^3)\cup_{[\iota_2,\iota_3]}e^5$.
        \begin{lem}
          If $M$ is nonspin, then $\pi_6(M^{(5)})\cong\mathbb{Z}\{\beta_0\} \oplus\mathbb{Z}_{12}\{a_3\}\oplus\mathbb{Z}_2\{p\eta_5\}$, where $\beta_0$ is the attaching map of the top cell of $\mathbb{C}P^2\times S^3$ equipped with the standard cell structure.
        \end{lem}
        The proof of this lemma is analogous to that of Lemma \ref{thm:12} and so is omitted.

Summarizing what we have:

\begin{prop}
  Let $M$ be a closed, smooth, $1$-connected $7$-manifold with integral cohomology ring isomorphic to $H^*(\mathbb{C}P^2\times S^3)$. Then $$M\simeq (\mathbb{C}P^2\vee S^3)\cup_{[\iota_2,\iota_3]+\delta\eta_3}e^5\cup_{\beta_{\delta}+ka_3+\epsilon p\eta_5}e^7$$ for some $\delta, \epsilon\in \{0,1\}$ and $0\leq k\leq 11$, where $p$ is the Hopf fibration and $\beta_0, \beta_1$ is the attaching map of the top cell of $\mathbb{C}P^2\times S^3, M'_{0,0}$, respectively. Moreover, $\delta=0$ if and only if $M$ is nonspin.
\end{prop}

We denote the right-hand side of the above formula as $X_{\delta,k,\epsilon}$ and $(\mathbb{C}P^2\vee S^3)\cup_{[\iota_2,\iota_3]+\delta\eta_3}e^5$ as $X_{\delta}$. It is easy to see that $X_{\delta,k,\epsilon}$ is a $1$-connected $7$-dimensional Poincar\'{e} complex with integral cohomology ring isomorphic to $H^*(\mathbb{C}P^2\times S^3)$. Conversely, any $1$-connected $7$-dimensional Poincar\'{e} complex whose integral cohomology ring is isomorphic to $H^*(\mathbb{C}P^2\times S^3)$ must be homotopy equivalent to some $X_{\delta,k,\epsilon}$.

\begin{lem}\label{thm:2}
  Two spaces $X_{\delta,k,\epsilon}$ and $X_{\delta,k',\epsilon'}$ are homotopy equivalent if and only if there exists $f\in \mathcal{E}(X_{\delta})$ such that $f_*(\beta_{\delta}+ka_3+\epsilon p\eta_5)=\pm(\beta_{\delta}+k'a_3+\epsilon' p\eta_5)$, where $\mathcal{E}(X)$ denotes the group of self homotopy equivalences of $X$.
\end{lem}

\begin{proof}
  \lq\lq$\Rightarrow$" Let $g:X_{\delta,k,\epsilon}\to X_{\delta,k',\epsilon'}$ be a homotopy equivalence. We may assume that $g$ is cellular. Let $f$ be the restriction of $g$ on the $5$-skeleton. As $f$ induces isomorphisms on $H_*(X_{\delta})$, it follows that $f\in \mathcal{E}(X_{\delta})$. We have the following commutative diagram:
  \[
  \begin{tikzcd}
  \mathbb{Z}\cong\pi_7(X_{\delta,k,\epsilon},X_{\delta}) \arrow[r, "\partial"] \arrow[d, "g_*","\cong"'] & \pi_6(X_{\delta}) \arrow[d,"f_*","\cong"'] \\
  \mathbb{Z}\cong\pi_7(X_{\delta,k',\epsilon'},X_{\delta}) \arrow[r,"\partial"] & \pi_6(X_{\delta})
  \end{tikzcd}
  \]
  Since $\partial$ sends $1$ to the attaching maps, the \lq\lq only if\rq\rq\ part follows immediately.

  \lq\lq$\Leftarrow$" The condition implies that $f$ can be extended to $g:X_{\delta,k,\epsilon}\to X_{\delta,k',\epsilon'}$. It can be easily checked that $g$ induces an isomorphism between the cohomology rings, which implies that $g$ is a homotopy equivalence.
\end{proof}

\begin{prop}
  Two spaces $X_{1,k,\epsilon}$ and $X_{1,k',\epsilon'}$ are homotopy equivalent if and only if $k\equiv k' \mod 6$.
\end{prop}

\begin{proof}
  Consider the following map:
  \[
  f_{l,c}:=\id^{lp+c\eta_3\eta_4}: X_1\xrightarrow{\psi} X_1\vee S^5 \xrightarrow{(\id, lp+c\eta_3\eta_4)} X_1.
  \]
  As $f_{l,c}$ induces an isomorphism on $H_*(X_1)$, it follows that $f_{l,c}\in \mathcal{E}(X_1)$. We analyze the effect of ${f_{l,c}}_*$ on $\pi_6(X_1)$.

    We have
    \[
    f_{l,c}\beta_1: S^6 \xrightarrow{\beta_1} X_1 \xrightarrow{\psi} X_1\vee S^5 \xrightarrow{(\id, lp+c\eta_3\eta_4)} X_1.
    \]
    Using the split short exact sequence (cf. \cite[Theorem XI.1.5]{Whi78})
    \[
    0\to \pi_7(X_1\times S^5,X_1\vee S^5) \to \pi_6(X_1\vee S^5) \to \pi_6(X_1\times S^5) \to 0,
    \]
    it is easy to see that $\pi_6(X_1\vee S^5)\cong \pi_6(X_1)\oplus\mathbb{Z}_2\{\eta_5\}\oplus\mathbb{Z}\{[\iota_2,\iota_5]\}$.
    \begin{clm}
    $\psi\beta_1=\beta_1\pm [\iota_2,\iota_5]$.
    \end{clm}

    \begin{proof}
      Suppose $\psi\beta_1=\alpha+a\eta_5+b[\iota_2,\iota_5]$, where $\alpha\in\pi_6(X_1)$. Then $\alpha=(\id, 0)\psi\beta_1=\beta_1$ ($0$ denotes the constant map):
      \[
      \begin{tikzcd}
      \alpha : S^6 \arrow[r,"\beta_1"] & X_1 \arrow[r,"\psi"]  \arrow[rr, bend right, "\simeq\id"] & X_1\vee S^5 \arrow[r,"{(\id,0)}"] & X_1 .
      \end{tikzcd}
      \]

     Similarly, we have $a\eta_5=(0,\id)\psi\beta_1=q\beta_1$, where $q:X_1\to S^5$ is the map collapsing the $4$-skeleton of $X_1$ to a point. It follows that $X_{1,k,\epsilon}/X_{1,k,\epsilon}^{(4)}\simeq S^5\cup_{a\eta_5}e^7$. Similar argument as in the proof of Lemma \ref{thm:1} shows that $Sq^2=0:H^5(X_{1,k,\epsilon};\mathbb{Z}_2)\to H^7(X_{1,k,\epsilon};\mathbb{Z}_2)$, which implies that $Sq^2=0:H^5(X_{1,k,\epsilon}/X_{1,k,\epsilon}^{(4)};\mathbb{Z}_2)\to H^7(X_{1,k,\epsilon}/X_{1,k,\epsilon}^{(4)};\mathbb{Z}_2)$. Therefore $a\eta_5=0$.

     Now let $X'=(X_1\vee S^5)\cup_{\psi\beta_1}e^7$. Then $\psi$ can be extended to a map $\Psi:X_{1,0,0}\to X'$. With the help of $\Psi$, one can deduce the cup product structure of $X'$ from $X_{1,0,0}$. It is not hard to obtain that $H^*(X')\cong\mathbb{Z}[x,y,s]/(x^3,y^2,s^2,xs-x^2y,ys)$, which forces $b$ to be $\pm 1$.
    \end{proof}

    Therefore,
    \[
    \begin{split}
       f_{l,c}\beta_1 & = (\id,lp+c\eta_3\eta_4)(\beta_1\pm [\iota_2,\iota_5]) \\
         & =\beta_1\pm[\iota_2,lp+c\eta_3\eta_4] \\
         & =\beta_1+l[\iota_2,p]+c[\iota_2,\eta_3\eta_4]
    \end{split}
    \]
    By \cite[p. 818]{BJS60}, we have $[\iota_2,p]=p\eta_5$. To calculate $[\iota_2,\eta_3\eta_4]$, recall the following formula:
    \begin{lem}[{cf. \cite[Theorem 8.18]{Whi78}}]
      Let $\alpha\in\pi_{p+1}(Y)$, $\beta\in\pi_{q+1}(Y)$, $\gamma\in\pi_m(S^p)$, and $\delta\in\pi_n(S^q)$. Then
      \[
      [\alpha\circ\Sigma\gamma,\beta\circ\Sigma\delta]=[\alpha,\beta]\circ\Sigma(\gamma\wedge\delta).
      \]
    \end{lem}
    Thus $[\iota_2,\eta_3\eta_4]=[\iota_2,\iota_3]\eta_4\eta_5$. Since $[\iota_2,\iota_3]=\eta_3\in\pi_4(X_1)$, it follows that $[\iota_2,\eta_3\eta_4]=\eta_3\eta_4\eta_5=6a_3$. Summarizing, we have
    \[
    {f_{l,c}}_*\beta_1=\beta_1+lp\eta_5+6ca_3.
    \]
    It is clear that ${f_{l,c}}_*a_3=a_3$ and ${f_{l,c}}_*p\eta_5=p\eta_5$, as $\mathbb{C}P^2$ and $S^3$ are fixed by $f_{l,c}$. Therefore,
  \[
  {f_{l,c}}_*(\beta_1+ka_3+\epsilon p\eta_5)=\beta_1+(6c+k)a_3+(l+\epsilon)p\eta_5.
  \]
  Combined with Lemma \ref{thm:2}, the \lq\lq if\rq\rq\ part follows.

  Since Theorem \ref{thm:3} implies that there are at least $6$ different $X_{1,k,\epsilon}$ up to homotopy equivalence, the \lq\lq only if\rq\rq\ part is also clear. This completes the proof of the proposition.
\end{proof}

The classification theorem now follows directly.

\begin{thm}\label{thm:4}
  Let $M$ be a closed, smooth, spin, $1$-connected $7$-manifold with integral cohomology ring isomorphic to $H^*(\mathbb{C}P^2\times S^3)$. Then $M$ is homotopy equivalent to some $M'_{k,0}$ (notation as in Section \ref{sec:2}). Moreover, its homotopy type is uniquely determined by $p_1(M) \mod 24$.
\end{thm}

\begin{rmk}
  For the nonspin case, we have only obtained a classification result of Poincar\'{e} complexes. We can prove that $X_{0,k,\epsilon}\simeq X_{0,k',\epsilon'}$ if and only if $k\equiv k' \mod 6$ and $\epsilon=\epsilon'$. Since the proof is rather complicated and not relevant to the rest of the paper, we shall not include it here. We know that $X_{0,k,0}$ is homotopy equivalent to some $S^3$-bundle over $\mathbb{C}P^2$ and so is smoothable. We guess that $X_{0,k,1}$ is not homotopy equivalent to any topological manifold.
\end{rmk}

\subsection{Classification up to homeomorphism, PL-homeomorphism and connected sum with homotopy spheres}

\begin{thm}\label{thm:5}
  Let $M$ be a closed, smooth, spin, $1$-connected $7$-manifold with integral cohomology ring isomorphic to $H^*(\mathbb{C}P^2\times S^3)$. Then $M$ is PL-homeomorphic to some $M'_{k,0}$, and its homeomorphism and PL-homeomorphism type are uniquely determined by $p_1(M)$. Furthermore, there exists a homotopy $7$-sphere $\Sigma^7$ such that $M$ is diffeomorphic to $M'_{k,0}\#\Sigma^7$.
\end{thm}

We use surgery theory to do the classification. The crucial part is the following lemma which is similar to \cite [Lemma 5.3] {CE03}:

\begin{lem}\label{thm:6}
  Let $f_m: M'_{k+6m,0}\to M'_{k,0}$ be a fiber homotopy equivalence. Then
  $$
  \eta([f_m])=m\in[M'_{k,0},G/\pl]\cong\mathbb{Z}.
  $$
\end{lem}

\begin{proof}
  The lemma contains two parts. One is $[M'_{k,0},G/\pl]\cong\mathbb{Z}$ and the other is $\eta([f_m])=m$. We prove them one by one.

  It is known that $[\mathbb{C}P^2,G/\pl]\cong\mathbb{Z}$ (cf. \cite[Lemma 14C.1]{Wal69}). Let $\sigma_k:\mathbb{C}P^2\to M'_{k,0}$ be a cross section. If we can prove that $\sigma_k^*:[M'_{k,0},G/\pl]\to [\mathbb{C}P^2,G/\pl]$ is an isomorphism, then the first part of the lemma follows.

  It has been proved in \cite{BV68} that $G/PL$ is an infinite loop space. Suppose $G/\pl\simeq \Omega^N X$ with $N$ sufficiently large. Then
  \[
  [M'_{k,0},G/\pl] \cong [M'_{k,0},\Omega^N X] \cong [\Sigma^N M'_{k,0},X].
  \]
  As shown in Section \ref{sec:3}, we have $M'_{k,0}\simeq (\mathbb{C}P^2\vee S^3)\cup_{[\iota_2,\iota_3]+\eta_3}e^5\cup e^7$. Therefore,
  \[
  \begin{split}
    \Sigma^N M'_{k,0} & \simeq (\Sigma^N \mathbb{C}P^2\vee \Sigma^{N+1}\mathbb{C}P^2)\cup e^{N+7} \\
      & \simeq \Sigma^N \mathbb{C}P^2\vee (\Sigma^{N+1}\mathbb{C}P^2\cup e^{N+7}) \rel \Sigma^N\mathbb{C}P^2,
  \end{split}
  \]
  as the suspension of a Whitehead product is always $0$ and $\pi_{N+6}(\Sigma^N\mathbb{C}P^2)=0$ (cf. \cite{Muk82}). Thus we have the following commutative diagram:
  \[
  \begin{tikzcd}[column sep=tiny]
  {[M'_{k,0},G/\pl]} \arrow[r,"\cong"] \arrow[d,"\sigma_k^*"] & {[\Sigma^N \mathbb{C}P^2\vee (\Sigma^{N+1}\mathbb{C}P^2\cup e^{N+7}),X]} \arrow[r,"\cong"] \arrow[d,"(\Sigma^N\sigma_k)^*"] & {[\Sigma^N\mathbb{C}P^2,X]}\times {[\Sigma^{N+1}\mathbb{C}P^2\cup e^{N+7},X]} \arrow[dl,"p_1"] \\
  {[\mathbb{C}P^2,G/\pl]} \arrow[r,"\cong"] & {[\Sigma^N\mathbb{C}P^2,X]} &
  \end{tikzcd}
  \]
  One can easily show that $[\Sigma^{N+1}\mathbb{C}P^2\cup e^{N+7},X]=0$ with the help of the Puppe exact sequence and the fact that $\pi_{2i+1}(G/\pl)=0$. Then it is clear that $\sigma_k^*$ is an isomorphism. Note that if we denote the bundle projection of $\xi'_{k,0}$ as $\pi_k$, then $\pi_k^*:[\mathbb{C}P^2,G/\pl]\to [M'_{k,0},G/\pl]$ is also an isomorphism.

  Next, we turn to prove the second part of the lemma. Let $W_k=D(\xi'_{k,0})$, which is the total space of the associated disk bundle of $\xi'_{k,0}$. Then $\partial W_k=M'_{k,0}$. Let $i_k:M'_{k,0}\to W_k$ be the inclusion. Note that $W_k$ can also be regarded as the mapping cylinder of $\pi_k: M'_{k,0}\to\mathbb{C}P^2$. Therefore $f_m$ can be extended to a fiber homotopy equivalence of pairs:
  \[
  (F_m,f_m):(W_{k+6m},M'_{k+6m,0})\xrightarrow{\simeq} (W_k,M'_{k,0}).
  \]
  We have the following commutative diagram:
  \[
  \begin{tikzcd}
  \mathscr{S}^{\pl}(W_k) \arrow[r,"\eta"] \arrow[d,"i_k^*"] & {[W_k,G/\pl]} \arrow[d,"i_k^*"] & {[\mathbb{C}P^2,G/\pl]} \arrow[l,"\pi_k^*"',"\cong"] \arrow[ld,"\pi_k^*","\cong"'] \\
  \mathscr{S}^{\pl}(M'_{k,0}) \arrow[r,"\eta"] &{[M'_{k,0},G/\pl]}
  \end{tikzcd}
  \]
  Therefore,
  \[
  \eta([f_m])=\eta i_k^*([F_m])=i_k^*\eta([F_m]).
  \]
  As we have already seen that $i_k^*:[W_k,G/\pl]\to[M'_{k,0},G/\pl]$ is an isomorphism, we only need to prove that $\eta([F_m])=m\in[W_k,G/\pl]\cong\mathbb{Z}$.

  We use the following commutative diagram:
  \[
  \begin{tikzcd}
  {[W_k,G/\pl]} \arrow[r,"j_*"] \arrow[d,"\omega_k^*","\cong"'] & {[W_k,B\pl]} \arrow[d,"\omega_k^*","\cong"'] \\
  {[\mathbb{C}P^2,G/\pl]} \arrow[r,"j_*"] & {[\mathbb{C}P^2,B\pl]}
  \end{tikzcd}
  \]
  Here $\omega_k:\mathbb{C}P^2\to W_k$ is the zero-section and $j:G/\pl\to B\pl$ is the obvious map.
  \begin{clm}
  $[\mathbb{C}P^2,B\pl]\cong\mathbb{Z}$ and $j_*$ is multiplication by $24$.
  \end{clm}

  \begin{proof}
  The statement $[\mathbb{C}P^2,B\pl]\cong\mathbb{Z}$ follows from the $6$-connectedness of $\pl/O$ and the well-known fact that $[\mathbb{C}P^2,BO] \cong \mathbb{Z}$ (1-1 correspondence by $p_1$). The statement for $j_*$ follows by the following exact ladder of Puppe sequences:
  \[
  \begin{tikzcd}[row sep=small]
   & \mathbb{Z} \arrow[d, equal] & \mathbb{Z} \arrow[d,equal] & \mathbb{Z}_2 \arrow[d,equal] & \\
  0 \arrow[r] & {[S^4,G/\pl]} \arrow[r,"\times 2"] \arrow[dd,"j_*","\times 24"'] & {[\mathbb{C}P^2,G/\pl]} \arrow[r] \arrow[dd,"j_*"] & {[S^2,G/\pl]} \arrow[r] \arrow[dd,"j_*"] & 0 \\
  & & & & \\
  0 \arrow[r] & {[S^4,B\pl]} \arrow[r,"\times 2"] \arrow[d,equal] & {[\mathbb{C}P^2,B\pl]} \arrow[r] \arrow[d,equal] & {[S^2,B\pl]} \arrow[r] \arrow[d,equal] & 0 \\
   & \mathbb{Z} & \mathbb{Z} & \mathbb{Z}_2 &
  \end{tikzcd}
  \]
  Note that the left vertical arrow can be easily calculated by the exact sequence of homotopy groups of the fibration $G/\pl\to B\pl\to BG$.
  \end{proof}

  Now we have
  \[
  \begin{split}
     j_*(\eta([F_m])) & = \nu(W_k)-{F_k^{-1}}^*(\nu(W_{k+6m})) \\
       & = -TW_k+{F_k^{-1}}^*(TW_{k+6m}) \\
       & = -(\pi_k^*T\mathbb{C}P^2+\pi_k^*\xi'_{k,0})+ {F_k^{-1}}^*(\pi_{k+6m}^*T\mathbb{C}P^2+\pi_{k+6m}^*\xi'_{k+6m,0}) \\
       & = -(\pi_k^*T\mathbb{C}P^2+\pi_k^*\xi'_{k,0}) +(\pi_k^*T\mathbb{C}P^2+\pi_k^*\xi'_{k+6m,0}) \\
       & = \pi_k^*\xi'_{k+6m,0}-\pi_k^*\xi'_{k,0}
  \end{split}
  \]
  Therefore,
  \[
  \begin{split}
     \omega_k^*j_*(\eta([F_m])) & =\omega_k^*(\pi_k^*\xi'_{k+6m,0}-\pi_k^*\xi'_{k,0}) \\
       & =\xi'_{k+6m,0}-\xi'_{k,0} \\
       & =4(k+6m)+1-(4k+1) \\
       & =24m
  \end{split}
  \]
  Then $\eta([F_m])=m$ follows. The proof of the lemma is now complete.
\end{proof}

\begin{proof}[Proof of theorem \ref{thm:5}]
  As already shown in Theorem \ref{thm:4}, $M$ is homotopy equivalent to some $M'_{k,0}$. The first part of the theorem now follows immediately from Lemma \ref{thm:6}, Theorem \ref{thm:3}, and the injectivity of $\eta$.

  For the second part, we use the Puppe sequence
  \[
  [S^7,\pl/O]\xrightarrow{q^*} [M'_{k,0},\pl/O]\to [{M'_{k,0}}^{(5)},\pl/O].
  \]
  Since $\pl/O$ is $6$-connected, we have $[{M'_{k,0}}^{(5)},\pl/O]=0$. Note that $q^*$ coincides with the map $\Theta_7\to\mathcal{S}^{\pl/O}(M'_{k,0})$ defined by $\Sigma^7\mapsto M'_{k,0}\#\Sigma^7$. Therefore, if $M$ is a closed, spin, $1$-connected $7$-manifold with integral cohomology ring isomorphic to $H^*(\mathbb{C}P^2\times S^3)$, then there exists some homotopy $7$-sphere $\Sigma^7$ such that $M$ is diffeomorphic to $M'_{k,0}\# \Sigma^7$. This completes our proof.
\end{proof}

Combining Theorem \ref{thm:4} and Theorem \ref{thm:5}, Theorem \ref{thm:10} follows.

\section{Certain 7-manifolds with metrics of positive Ricci curvature} \label{sec:6}

In this section, we focus on the following:

\theoremstyle{plain}
\newtheorem*{numlessthm1}{Theorem \ref{thm:7}}

\begin{numlessthm1}
Let $M$ be a closed, smooth, spin, $1$-connected $7$-manifold with integral cohomology ring isomorphic to $ H^*(\mathbb{C}P^2\times S^3)$ or $H^*(S^2\times S^5)$. Then $M$ admits a Riemannian metric with positive Ricci curvature.
\end{numlessthm1}

The proof relies on the surgery theorem of Wraith. To state it, we first recall some conventions on plumbing. Plumbing can be represented by a labelled graph. For each bundle, we use a labelled vertex to represent it. If two bundles are plumbed together, we connect the corresponding vertices with an edge. Readers who are not familiar with the notion of plumbing can find a great introduction in \cite{CW17}, and we will adopt their terminologies.

Wraith's surgery theorem can be stated as follows:

\begin{thm}[\cite{Wra97}]\label{thm:11}
  Let $T$ be a labelled tree with one vertex a linear $D^n$-bundle over an $n$-dimensional Ricci positive manifold and others linear $D^n$-bundles over $S^n$. Then the boundary of the plumbing manifold admits a metric with positive Ricci curvature.
\end{thm}

\begin{rmk}
  In \cite{Wra97}, Wraith proved the above theorem where the base spaces of the bundles are all
  spheres. However, for the case where one bundle has a base space with positive Ricci curvature, Wraith informs us that the proof of the theorem still applies.
\end{rmk}

To simplify our problem, we also need a useful lemma:

\begin{lem}[\cite{CW17}]\label{thm:8}
  Let $W_1$, $W_2$ be two manifolds obtained by tree-like plumbing with $D^n$-bundles over $n$-manifolds. Connect $W_1$ and $W_2$ by the following plumbing diagram to obtain $W$:
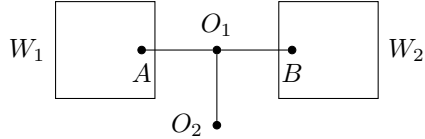
  \begin{center}
    \begin{tikzpicture}[p/.style={circle, draw, fill, inner sep=0pt, minimum size=1mm}, i/.style={draw}, o/.style={circle, fill=white, inner sep=0pt}]
    \path (0,0) node[o] (a1) {}
    (0.5,0.5) node[o] (a2) {}
    (0.5,-0.5) node[o] (a3) {}
    (1,0) node[p, label=below:$A$] (a) {}
    (2,0) node[p, label=above:$O_1$] (o1) {}
    (2,-1) node[p, label=left:$O_2$] (o2) {}
    (3,0) node[p,label=below:$B$] (b) {}
    (3.5,0.5) node[o] (b2) {}
    (3.5,-0.5) node[o] (b3) {}
    (4,0) node[o] (b1) {};

    \draw (a) -- (o1) -- (b);
    \draw (o1) -- (o2);
    \node[i, fit=(a1)(a2)(a3)(a), label=left:$W_1$] {};
    \node[i, fit=(b1)(b2)(b3)(b), label=right:$W_2$] {};
    \end{tikzpicture}
  \end{center}
  Here $O_1$, $O_2$ are trivial $D^n$-bundles over $S^n$, i.e. $S^n\times D^n$, each box contains a tree producing $W_i$, and $A$, $B$ are arbitrary points in the trees. Then $\partial W\cong \partial W_1\#\partial W_2$.
\end{lem}

\begin{proof}[Proof of theorem \ref{thm:7}]
  Our strategy is to prove that $M$ is the boundary of some tree-like plumbing of a $D^4$-bundle over $\mathbb{C}P^2$ and several $D^4$-bundles over $S^4$. By Lemma \ref{thm:8}, we only need to prove it up to connecting sum with a homotopy $7$-sphere.

  The case of $H^*(M)\cong H^*(\mathbb{C}P^2\times S^3)$ is an immediate consequence of Theorem \ref{thm:5}.

  For the case $H^*(M)\cong H^*(S^2\times S^5)$, we first prove the following:
  \begin{clm}
  There are at most $24$ different $M$ up to homeomorphism, and at most $24\times 28$ up to diffeomorphism, where the $28$ comes from connecting sum with $28$ homotopy $7$-spheres.
  \end{clm}

  \begin{proof}
    First, we observe that $M\cong S^2\times D^5\cup_f S^2\times D^5$ with $f\in \diff(S^2\times S^4)$. To see this, notice that the generator of $H_2(M)\cong\mathbb{Z}$ can be represented by an embedding $h_1:S^2\hookrightarrow M$. The condition that $M$ is spin implies that the normal bundle of $h_1$ is trivial, which indicates that $h_1$ can extend to an embedding $h_1: S^2\times D^5\hookrightarrow M$. It is easy to see that $M-h_1(\Int (S^2\times D^5))$ is 1-connected, spin, and has homology groups isomorphic to those of $S^2$. Thus the generator of its second homology group can also be represented by an embedding $h_2:S^2\times D^5\hookrightarrow M-h_1(\Int (S^2\times D^5))$. Now notice that the complement of $h_1$ and $h_2$ in $M$ is an h-cobordism, and the observation follows by applying the h-cobordism theorem.

    Let $r_n:S^n\to S^n$ be a degree $-1$ map on $S^n$. As gluing two copies of $S^2\times D^5$ using $(r_2\times \id_{S^4})f$ or $(\id_{S^2}\times r_4)f$ gives diffeomorphic manifold as using $f$, we can always assume that $f$ induces identity on homology. Also, notice that the diffeomorphism type of $M$ only depends on the pseudo-isotopy class of $f$. Recall that $f_0,f_1\in \diff(N)$ are said to be pseudo-isotopic if there exists $F\in\diff(N\times I)$ such that $F(x,0)=(f_0(x),0)$ and $F(x,1)=(f_1(x),1)$ for all $x\in N$. Therefore, if we denote by $\widetilde{\pi}_0(\sdiff N)$ the set of pseudo-isotopy classes of diffeomorphisms of $N$ which induce identity on homology, then the number of different diffeomorphism types of such $M$ is at most the order of $\widetilde{\pi}_0(\sdiff S^2\times S^4)$. By \cite[Theorem II]{Sat69}, we have $\widetilde{\pi}_0(\sdiff S^2\times S^4)\cong FC_4^3\oplus\Theta_7$, where $FC_4^3$ is the group of concordance classes of framed embeddings of $S^4$ into $S^7$ and has order $24$ \cite{Hae66}. Hence the claim follows.
  \end{proof}

  Now our result will follow if we can construct $24$ nonhomeomorphic $M$. Notations as in Section \ref{sec:2}. Let $W_{k,l,m}=D(\xi'_{m,0})\Box D(k\alpha+l\beta)$, where $\Box$ is the symbol for plumbing. Then $M_{k,l,m}=\partial W_{k,l,m}$ is a $1$-connected spin $7$-manifold with $H^*(M_{k,l,m})\cong H^*(S^2\times S^5)$. Easy calculation shows that the Kreck-Stolz invariants \cite{KS88} of $M_{k,l,m}$ are
  \begin{align*}
    s_1(M_{k,l,m}) & = \frac{1}{56} (m+1)(2k-2l-lm), \\
    s_2(M_{k,l,m}) & = \frac{1}{12} (lm+l-k), \\
    s_3(M_{k,l,m}) & = \frac{1}{6} (2lm-2k-l).
  \end{align*}
  Let $a=lm+l-k$ and replace $k$ by $lm+l-a$ in the results, we have
  \begin{align*}
    s_1(M_{k,l,m}) & = \frac{1}{56} (m+1)(lm-2a), \\
    s_2(M_{k,l,m}) & = \frac{1}{12} a, \\
    s_3(M_{k,l,m}) & = \frac{1}{3} a-\frac{1}{2} l.
  \end{align*}
  Since $s_2,s_3$ are homeomorphism invariants and $a,l,m$ can be arbitrary integers, our result follows.
\end{proof}

\section*{Acknowledgements}

The author would like to thank Yang Su and Yi Jiang for many helpful discussions and David Wraith's helpful communications about his surgery theorem. The referee is also gratefully acknowledged for pointing out a mistake in an earlier version of the paper and providing detailed comments and suggestions, which significantly improved the presentation of the paper.

\bibliography{bib.bib}

\end{document}